\documentclass[fleqn,reqno,a4paper]{amsart}
\usepackage{amsmath}
\usepackage{amsfonts}

\newtheorem{theorem}{Theorem}
\theoremstyle{plain}

\newtheorem{corollary}{Corollary}

\newtheorem{definition}{Definition}

\newtheorem{lemma}{Lemma}

\newtheorem{remark}{Remark}

\numberwithin{equation}{section}
\input{tcilatex}

\begin{document}
\title[New general integral inequalities]{New general integral inequalities
for $(\alpha ,m)$--GA-convex functions via Hadamard fractional integrals}
\author{\.{I}mdat \.{I}\c{s}can}
\address{Department of Mathematics, Faculty of Sciences and Arts, Giresun
University, Giresun, Turkey}
\email{imdat.iscan@giresun.edu.tr}
\author{Mehmet Kunt}
\address{Department of Mathematics, Faculty of Sciences, Karadeniz Technical
University, Trabzon, Turkey}
\subjclass[2000]{ 26A51, 26A33, 26D15. }
\keywords{Hermite--Hadamard type inequality, Ostrowski type inequality,
Simpson type inequality, GA-$(\alpha ,m)$-convex function.}

\begin{abstract}
In this paper, the authors gives a new identity for Hadamard fractional
integrals. By using of this identity, the authors obtains new estimates on
generalization of Hadamard, Ostrowski and Simpson type inequalities for $%
(\alpha ,m)$-GA-convex function via Hadamard fractional integral.
\end{abstract}

\maketitle

\section{Introduction}

Let a real function $f$ be defined on some nonempty interval $I$ of real
line $%
\mathbb{R}
$. The function $f$ is said to be convex on $I$ if inequality%
\begin{equation*}
f(tx+(1-t)y)\leq tf(x)+(1-t)f(y)
\end{equation*}%
holds for all $x,y\in I$ and $t\in \left[ 0,1\right] .$

Following inequalities are well known in the literature as Hermite-Hadamard
inequality, Ostrowski inequality and Simpson inequality respectively:

\begin{theorem}
Let $f:I\subseteq \mathbb{R\rightarrow R}$ be a convex function defined on
the interval $I$ of real numbers and $a,b\in I$ with $a<b$. The following
double inequality holds:%
\begin{equation*}
f\left( \frac{a+b}{2}\right) \leq \frac{1}{b-a}\dint\limits_{a}^{b}f(x)dx%
\leq \frac{f(a)+f(b)}{2}\text{.}
\end{equation*}
\end{theorem}

\begin{theorem}
Let $f:I\subseteq \mathbb{R\rightarrow R}$ be a mapping differentiable in $%
I^{\circ },$ the interior of I, and let $a,b\in I^{\circ }$ with $a<b.$ If $%
\left\vert f^{\prime }(x)\right\vert \leq M,$ $x\in \left[ a,b\right] ,$
then we the following inequality holds:%
\begin{equation*}
\left\vert f(x)-\frac{1}{b-a}\dint\limits_{a}^{b}f(t)dt\right\vert \leq 
\frac{M}{b-a}\left[ \frac{\left( x-a\right) ^{2}+\left( b-x\right) ^{2}}{2}%
\right]
\end{equation*}%
for all $x\in \left[ a,b\right] .$ The constant $\frac{1}{4}$ is the best
possible in the sense that it cannot be replaced by a smaller one.
\end{theorem}

\begin{theorem}
Let $f:\left[ a,b\right] \mathbb{\rightarrow R}$ be a four times
continuously differentiable mapping on $\left( a,b\right) $ and $\left\Vert
f^{(4)}\right\Vert _{\infty }=\underset{x\in \left( a,b\right) }{\sup }%
\left\vert f^{(4)}(x)\right\vert <\infty .$ Then the following inequality
holds:%
\begin{equation*}
\left\vert \frac{1}{3}\left[ \frac{f(a)+f(b)}{2}+2f\left( \frac{a+b}{2}%
\right) \right] -\frac{1}{b-a}\dint\limits_{a}^{b}f(x)dx\right\vert \leq 
\frac{1}{2880}\left\Vert f^{(4)}\right\Vert _{\infty }\left( b-a\right) ^{4}.
\end{equation*}
\end{theorem}

The following defnitions are well known in the literature.

\begin{definition}
\cite{10,11}. A function $f:I\subseteq \left( 0,\infty \right) \rightarrow 
\mathbb{R}
$ is said to be GA-convex (geometric-arithmatically convex) if%
\begin{equation*}
f(x^{t}y^{1-t})\leq tf(x)+\left( 1-t\right) f(y)
\end{equation*}%
for all $x,y\in I$ and $t\in \left[ 0,1\right] $.

\begin{definition}
\cite{8}. Let $f:(0,b]\rightarrow 
\mathbb{R}
,b>0,$ and $\left( \alpha ,m\right) \in \left( 0,1\right] ^{2}$. If 
\begin{equation*}
f(x^{t}y^{m\left( 1-t\right) })\leq t^{\alpha }f(x)+m\left( 1-t^{\alpha
}\right) f(y)
\end{equation*}%
for all $x,y\in (0,b]$ and $t\in \lbrack 0,1]$, then $f$ is said to be a $%
\left( \alpha ,m\right) $-GA-convex function .
\end{definition}
\end{definition}

\bigskip Note that $\left( \alpha ,m\right) \in \left\{ \left( 1,m\right)
,\left( 1,1\right) ,\left( \alpha ,1\right) \right\} $ one obtains the
following classes of functions: $m$-GA-convex, GA-convex, $\alpha $%
-GA-convex (or GA-$s$-convex in the first sense, if we take $s$ instead of $%
\alpha $ (see \cite{19})).

We will now give definitions of the right-sided and left-sided Hadamard
fractional integrals which are used throughout this paper.

\begin{definition}
\cite{4}. Let $f\in L\left[ a,b\right] $. The right-sided and left-sided
Hadamard fractional integrals $J_{a^{+}}^{\theta }f$ and $J_{b^{-}}^{\theta
}f$ of oder $\theta >0$ with $b>a\geq 0$ are defined by

\begin{equation*}
J_{a+}^{\theta }f(x)=\frac{1}{\Gamma (\theta )}\dint\limits_{a}^{x}\left(
\ln \frac{x}{t}\right) ^{\theta -1}f(t)\frac{dt}{t},\ a<x<b
\end{equation*}

and

\begin{equation*}
J_{b-}^{\theta }f(x)=\frac{1}{\Gamma (\theta )}\dint\limits_{x}^{b}\left(
\ln \frac{t}{x}\right) ^{\theta -1}f(t)\frac{dt}{t},\ a<x<b
\end{equation*}%
respectively, where $\Gamma (\theta )$ is the Gamma function defined by $%
\Gamma (\theta )=$ $\dint\limits_{0}^{\infty }e^{-t}t^{\theta -1}dt$.
\end{definition}

In \cite{20}, \.{I}\c{s}can represented Hermite-Hadamard's inequalities for
GA-convex functions in fractional integral forms as follows:

\begin{theorem}
Let $f:I\subseteq \left( 0,\infty \right) \rightarrow 
\mathbb{R}
$ be a function such that $f\in L[a,b]$, where $a,b\in I$ with $a<b$. If $f$
is a GA-convex function on $[a,b]$, then the following inequalities for
fractional integrals hold:%
\begin{equation*}
f\left( \sqrt{ab}\right) \leq \frac{\Gamma (\theta +1)}{2\left( \ln \frac{b}{%
a}\right) ^{\theta }}\left\{ J_{a+}^{\theta }f(b)+J_{b-}^{\theta
}f(a)\right\} \leq \frac{f(a)+f(b)}{2}
\end{equation*}%
with $\alpha >0$.
\end{theorem}

In \cite{20}, \.{I}\c{s}can gave the following identity for differentiable
functions..

\begin{lemma}
\label{Lemma1}Let $f:I\subseteq \left( 0,\infty \right) \rightarrow 
\mathbb{R}
$ be a differentiable function on $I^{\circ }$ such that $f^{\prime }\in
L[a,b]$, where $a,b\in I$ with $a<b$. Then for all $x\in \lbrack a,b]$ , $%
\lambda \in \left[ 0,1\right] $ and $\alpha >0$ we have:%
\begin{eqnarray*}
I_{f}\left( x,\lambda ,\theta ,a,b\right) &=&\left( 1-\lambda \right) \left[
\ln ^{\theta }\frac{x}{a}+\ln ^{\theta }\frac{b}{x}\right] f\left( x\right)
+\lambda \left[ f\left( a\right) \ln ^{\theta }\frac{x}{a}+f\left( b\right)
\ln ^{\theta }\frac{b}{x}\right] \\
&&-\Gamma (\theta +1)\left[ J_{x-}^{\theta }f(a)+J_{x+}^{\theta }f(b)\right]
\\
&=&a\left( \ln \frac{x}{a}\right) ^{\theta +1}\dint\limits_{0}^{1}\left(
t^{\theta }-\lambda \right) \left( \frac{x}{a}\right) ^{t}f^{\prime }\left(
x^{t}a^{1-t}\right) dt \\
&&-b\left( \ln \frac{b}{x}\right) ^{\theta +1}\dint\limits_{0}^{1}\left(
t^{\theta }-\lambda \right) \left( \frac{x}{b}\right) ^{t}f^{\prime }\left(
x^{t}b^{1-t}\right) dt.
\end{eqnarray*}
\end{lemma}

In recent years, many athors have studied errors estimations for
Hermite-Hadamard, Ostrowski and Simpson inequalities; for refinements,
counterparts, generalization see \cite{1,2,3,5,6,7,12,13,15,16,17,18}.

In this paper, new identity for fractional integrals have been defined. By
using of this identity, we obtained a generalization of Hadamard, Ostrowski
and Simpson type inequalities for $(\alpha ,m)$-GA-convex functions via
Hadamard fractional integrals.

\ 

\section{Main results}

Let $f:I\subseteq \left( 0,\infty \right) \rightarrow 
\mathbb{R}
$ be a differentiable function on $I^{\circ }$, the interior of $I$,
throughout this section we will take

\begin{eqnarray*}
K_{f}\left( \lambda ,\theta ,x^{m},a^{m},b^{m}\right) &=&\left( 1-\lambda
\right) m^{\theta }\left[ \ln ^{\theta }\frac{x}{a}+\ln ^{\theta }\frac{b}{x}%
\right] f(x^{m}) \\
&&+\lambda m^{\theta }\left[ f(a^{m})\ln ^{\theta }\frac{x}{a}+f(b^{m})\ln
^{\theta }\frac{b}{x}\right] \\
&&-\Gamma \left( \theta +1\right) \left[ J_{x^{m}-}^{\theta
}f(a^{m})+J_{x^{m}+}^{\theta }f(b^{m})\right]
\end{eqnarray*}%
where $a,b\in I$ with $a<b$, $\ x\in \lbrack a,b]$ , $\lambda \in \left[ 0,1%
\right] $, $\theta >0$ and $\Gamma $ is Euler Gamma function.

Similarly to Lemma \ref{Lemma1}, we can prove the following lemma.

\begin{lemma}
\label{Lemma2}Let $f:I\subseteq \left( 0,\infty \right) \rightarrow 
\mathbb{R}
$ be a differentiable function on $I^{\circ }$ such that $f^{\prime }\in
L[a^{m},b^{m}]$, where $a^{m},b\in I$ with $a<b$ and $m\in \left( 0,1\right] 
$. Then for all $x\in \lbrack a,b]$, $\lambda \in \left[ 0,1\right] $ and $%
\theta >0$ we have:%
\begin{eqnarray*}
&&K_{f}\left( \lambda ,\theta ,x^{m},a^{m},b^{m}\right) =m^{\theta
+1}a^{m}\left( \ln \frac{x}{a}\right) ^{\theta +1}\dint\limits_{0}^{1}\left(
t^{\theta }-\lambda \right) \left( \frac{x}{a}\right) ^{mt}f^{\prime }\left(
x^{mt}a^{m\left( 1-t\right) }\right) dt \\
&&-m^{\theta +1}b^{m}\left( \ln \frac{b}{x}\right) ^{\theta
+1}\dint\limits_{0}^{1}\left( t^{\theta }-\lambda \right) \left( \frac{x}{b}%
\right) ^{mt}f^{\prime }\left( x^{mt}b^{m\left( 1-t\right) }\right) dt.
\end{eqnarray*}
\end{lemma}

\begin{theorem}
\label{Theorem5}Let $f:I\subseteq \left( 0,\infty \right) \rightarrow 
\mathbb{R}
$ be a differentiable function on $I^{\circ }$ such that $f^{\prime }\in
L[a^{m},b^{m}]$, where $a^{m},b\in I$ $^{\circ }$ with $a<b$ and $m\in
\left( 0,1\right] $. If $|f^{\prime }|^{q}$ is $\left( \alpha ,m\right) $%
-GA-convex on $[a^{m},b]$ for some fixed $q\geq 1$, $x\in \lbrack a,b]$, $%
\lambda \in \left[ 0,1\right] $ and $\theta >0$ then the following
inequality for fractional integrals holds%
\begin{eqnarray}
&&\left\vert K_{f}\left( \lambda ,\theta ,x^{m},a^{m},b^{m}\right)
\right\vert \leq m^{\theta +1}C_{o}\left( \theta ,\lambda \right) ^{1-\frac{1%
}{q}}  \notag \\
&&\times \left\{ a^{m}\left( \ln \frac{x}{a}\right) ^{\theta +1}\left( 
\begin{array}{c}
\left\vert f^{\prime }\left( x^{m}\right) \right\vert ^{q}C_{1}\left(
x,\theta ,\lambda ,q,m,\alpha \right) \\ 
+m\left\vert f^{\prime }\left( a\right) \right\vert ^{q}C_{2}\left( x,\theta
,\lambda ,q,m,\alpha \right)%
\end{array}%
\right) ^{\frac{1}{q}}\right.  \notag \\
&&\left. +b^{m}\left( \ln \frac{b}{x}\right) ^{\theta +1}\left( 
\begin{array}{c}
\left\vert f^{\prime }\left( x^{m}\right) \right\vert ^{q}C_{3}\left(
x,\theta ,\lambda ,q,m,\alpha \right) \\ 
+m\left\vert f^{\prime }\left( b\right) \right\vert ^{q}C_{4}\left( x,\theta
,\lambda ,q,m,\alpha \right)%
\end{array}%
\right) ^{\frac{1}{q}}\right\}  \label{2.1}
\end{eqnarray}%
where 
\begin{eqnarray*}
C_{o}\left( \theta ,\lambda \right) &=&\frac{2\theta \lambda ^{1+\frac{1}{%
\theta }}+1}{\theta +1}-\lambda , \\
C_{1}\left( x,\theta ,\lambda ,q,m,\alpha \right)
&=&\dint\limits_{0}^{1}\left\vert t^{\theta }-\lambda \right\vert \left( 
\frac{x}{a}\right) ^{qmt}t^{\alpha }dt, \\
C_{2}\left( x,\theta ,\lambda ,q,m,\alpha \right)
&=&\dint\limits_{0}^{1}\left\vert t^{\theta }-\lambda \right\vert \left( 
\frac{x}{a}\right) ^{qmt}\left( 1-t^{\alpha }\right) dt, \\
C_{3}\left( x,\theta ,\lambda ,q,m,\alpha \right)
&=&\dint\limits_{0}^{1}\left\vert t^{\theta }-\lambda \right\vert \left( 
\frac{x}{b}\right) ^{qmt}t^{\alpha }dt, \\
C_{4}\left( x,\theta ,\lambda ,q,m,\alpha \right)
&=&\dint\limits_{0}^{1}\left\vert t^{\theta }-\lambda \right\vert \left( 
\frac{x}{b}\right) ^{qmt}\left( 1-t^{\alpha }\right) dt.
\end{eqnarray*}
\end{theorem}

\begin{proof}
From Lemma \ref{Lemma2}, property of the modulus and using the power-mean
inequality we have%
\begin{eqnarray}
&&\left\vert K_{f}\left( \lambda ,\theta ,x^{m},a^{m},b^{m}\right)
\right\vert \leq m^{\theta +1}\left( \dint\limits_{0}^{1}\left\vert
t^{\theta }-\lambda \right\vert dt\right) ^{1-\frac{1}{q}}  \notag \\
&&\times \left\{ a^{m}\left( \ln \frac{x}{a}\right) ^{\theta +1}\left(
\dint\limits_{0}^{1}\left\vert t^{\theta }-\lambda \right\vert \left( \frac{x%
}{a}\right) ^{qmt}\left\vert f^{\prime }\left( x^{mt}a^{m\left( 1-t\right)
}\right) \right\vert ^{q}dt\right) ^{\frac{1}{q}}\right.  \notag \\
&&\left. +b^{m}\left( \ln \frac{b}{x}\right) ^{\theta +1}\left(
\dint\limits_{0}^{1}\left\vert t^{\theta }-\lambda \right\vert \left( \frac{x%
}{b}\right) ^{qmt}\left\vert f^{\prime }\left( x^{mt}b^{m\left( 1-t\right)
}\right) \right\vert ^{q}dt\right) ^{\frac{1}{q}}\right\} .  \label{2.2}
\end{eqnarray}%
Since $|f^{\prime }|^{q}$ is $\left( \alpha ,m\right) $-GA-convex on $%
[a^{m},b],$ for all $t\in \left[ 0,1\right] $%
\begin{equation}
\left\vert f^{\prime }\left( x^{mt}a^{m\left( 1-t\right) }\right)
\right\vert ^{q}\leq t^{\alpha }\left\vert f^{\prime }\left( x^{m}\right)
\right\vert ^{q}+m\left( 1-t^{\alpha }\right) \left\vert f^{\prime }\left(
a\right) \right\vert ^{q},  \label{2.3}
\end{equation}%
\begin{equation}
\left\vert f^{\prime }\left( x^{mt}b^{m\left( 1-t\right) }\right)
\right\vert ^{q}\leq t^{\alpha }\left\vert f^{\prime }\left( x^{m}\right)
\right\vert ^{q}+m\left( 1-t^{\alpha }\right) \left\vert f^{\prime }\left(
b\right) \right\vert ^{q}.  \label{2.4}
\end{equation}

By a simple computation%
\begin{eqnarray}
\dint\limits_{0}^{1}\left\vert t^{\theta }-\lambda \right\vert dt
&=&\dint\limits_{0}^{\lambda ^{1/\theta }}\left( \lambda -t^{\theta }\right)
dt+\dint\limits_{\lambda ^{1/\theta }}^{1}\left( t^{\theta }-\lambda \right)
dt  \notag \\
&=&\frac{2\theta \lambda ^{1+\frac{1}{\theta }}+1}{\theta +1}-\lambda .
\label{2.5}
\end{eqnarray}%
If we use $\left( \ref{2.3}\right) $, $\left( \ref{2.4}\right) $ and $\left( %
\ref{2.5}\right) $ in $\left( \ref{2.2}\right) $, we obtain $\left( \ref{2.1}%
\right) $. This completes the proof.
\end{proof}

\begin{corollary}
Under the assumptions of Theorem \ref{Theorem5} with $q=1,$ the inequality $%
\left( \ref{2.1}\right) $ reduced to the following inequality%
\begin{eqnarray*}
&&K_{f}\left( \lambda ,\theta ,x^{m},a^{m},b^{m}\right) \leq m^{\theta +1} \\
&&\times \left\{ a^{m}\left( \ln \frac{x}{a}\right) ^{\theta +1}\left( 
\begin{array}{c}
\left\vert f^{\prime }\left( x^{m}\right) \right\vert C_{1}\left( x,\theta
,\lambda ,1,m,\alpha \right) \\ 
+m\left\vert f^{\prime }\left( a\right) \right\vert C_{2}\left( x,\theta
,\lambda ,1,m,\alpha \right)%
\end{array}%
\right) \right. \\
&&\left. +b^{m}\left( \ln \frac{b}{x}\right) ^{\theta +1}\left( 
\begin{array}{c}
\left\vert f^{\prime }\left( x^{m}\right) \right\vert C_{3}\left( x,\theta
,\lambda ,1,m,\alpha \right) \\ 
+m\left\vert f^{\prime }\left( b\right) \right\vert C_{4}\left( x,\theta
,\lambda ,1,m,\alpha \right)%
\end{array}%
\right) \right\}
\end{eqnarray*}
\end{corollary}

\begin{corollary}
Under the assumptions of Theorem \ref{Theorem5} with $x=\sqrt{ab}$, $\lambda
=\frac{1}{3}$ from the inequality $\left( \ref{2.1}\right) $ we get the
following Simpson type inequality for fractional integrals%
\begin{eqnarray*}
&&\frac{2^{\theta -1}}{\left( m\ln \frac{b}{a}\right) ^{\theta }}\left\vert
K_{f}\left( \frac{1}{3},\theta ,\left( \sqrt{ab}\right)
^{m},a^{m},b^{m}\right) \right\vert =\left\vert \frac{1}{6}\left[
f(a^{m})+4f\left( \left( \sqrt{ab}\right) ^{m}\right) +f(b^{m})\right]
\right. \\
&&-\frac{2^{\theta -1}\Gamma \left( \theta +1\right) }{\left( m\ln \frac{b}{a%
}\right) ^{\theta }}\left. \left[ J_{\left( \sqrt{ab}\right) ^{m}-}^{\theta
}f(a^{m})+J_{\left( \sqrt{ab}\right) ^{m}+}^{\theta }f(b^{m})\right]
\right\vert \leq \frac{m\ln \frac{b}{a}}{4}C_{0}^{1-\frac{1}{q}}\left(
\theta ,\frac{1}{3}\right) \\
&&\times \left\{ a^{m}\left[ 
\begin{array}{c}
\left\vert f^{\prime }\left( \left( \sqrt{ab}\right) ^{m}\right) \right\vert
^{q}C_{1}\left( \sqrt{ab},\theta ,\frac{1}{3},q,m,\alpha \right) \\ 
+m\left\vert f^{\prime }\left( a\right) \right\vert ^{q}C_{2}\left( \sqrt{ab}%
,\theta ,\frac{1}{3},q,m,\alpha \right)%
\end{array}%
\right] ^{\frac{1}{q}}\right. \\
&&\left. +b^{m}\left[ 
\begin{array}{c}
\left\vert f^{\prime }\left( \left( \sqrt{ab}\right) ^{m}\right) \right\vert
^{q}C_{3}\left( \sqrt{ab},\theta ,\frac{1}{3},q,m,\alpha \right) \\ 
+m\left\vert f^{\prime }\left( b\right) \right\vert ^{q}C_{4}\left( \sqrt{ab}%
,\theta ,\frac{1}{3},q,m,\alpha \right)%
\end{array}%
\right] ^{\frac{1}{q}}\right\}
\end{eqnarray*}
\end{corollary}

\begin{corollary}
\label{Corollary3}Under the assumptions of Theorem \ref{Theorem5} with $x=%
\sqrt{ab}$,$\ \lambda =0$ from the inequality $\left( \ref{2.1}\right) $ we
get the following midpoint-type inequality for fractional integrals%
\begin{eqnarray*}
&&\frac{2^{\theta -1}}{\left( m\ln \frac{b}{a}\right) ^{\theta }}\left\vert
K_{f}\left( 0,\theta ,\left( \sqrt{ab}\right) ^{m},a^{m},b^{m}\right)
\right\vert \\
&=&\left\vert f\left( \left( \sqrt{ab}\right) ^{m}\right) -\frac{2^{\theta
-1}\Gamma \left( \theta +1\right) }{\left( m\ln \frac{b}{a}\right) ^{\theta }%
}\left[ J_{\left( \sqrt{ab}\right) ^{m}-}^{\theta }f(a^{m})+J_{\left( \sqrt{%
ab}\right) ^{m}+}^{\theta }f(b^{m})\right] \right\vert \\
&\leq &\frac{m\ln \frac{b}{a}}{4}\left( \frac{1}{\theta +1}\right) ^{1-\frac{%
1}{q}}\left\{ a^{m}\left[ 
\begin{array}{c}
\left\vert f^{\prime }\left( \left( \sqrt{ab}\right) ^{m}\right) \right\vert
^{q}C_{1}\left( \sqrt{ab},\theta ,0,q,m,\alpha \right) \\ 
+m\left\vert f^{\prime }\left( a\right) \right\vert ^{q}C_{2}\left( \sqrt{ab}%
,\theta ,0,q,m,\alpha \right)%
\end{array}%
\right] ^{\frac{1}{q}}\right. \\
&&\left. +b^{m}\left[ 
\begin{array}{c}
\left\vert f^{\prime }\left( \left( \sqrt{ab}\right) ^{m}\right) \right\vert
^{q}C_{3}\left( \sqrt{ab},\theta ,0,q,m,\alpha \right) \\ 
+m\left\vert f^{\prime }\left( b\right) \right\vert ^{q}C_{4}\left( \sqrt{ab}%
,\theta ,0,q,m,\alpha \right)%
\end{array}%
\right] ^{\frac{1}{q}}\right\}
\end{eqnarray*}
\end{corollary}

\begin{remark}
If we take $\theta =1$, $m=1$ in Corollary \ref{Corollary3} we have the
following midpoint-type inequality for $\alpha $-GA-convex function (or
GA-s-convex function in the first sense), which is the same with the
inequality $\left( 9\right) $ of Theorem 3.4.b. in \cite{19}, 
\begin{eqnarray*}
&&\left\vert f\left( \sqrt{ab}\right) -\frac{1}{\ln b-\ln a}\int_{a}^{b}%
\frac{f\left( x\right) }{x}dx\right\vert \\
&\leq &\ln \frac{b}{a}\left( \frac{1}{2}\right) ^{3-\frac{1}{q}}\left\{ a%
\left[ 
\begin{array}{c}
\left\vert f^{\prime }\left( \sqrt{ab}\right) \right\vert ^{q}C_{1}\left( 
\sqrt{ab},1,0,q,1,\alpha \right) \\ 
+\left\vert f^{\prime }\left( a\right) \right\vert ^{q}C_{2}\left( \sqrt{ab}%
,1,0,q,1,\alpha \right)%
\end{array}%
\right] ^{\frac{1}{q}}\right. \\
&&\left. +b\left[ 
\begin{array}{c}
\left\vert f^{\prime }\left( \sqrt{ab}\right) \right\vert ^{q}C_{3}\left( 
\sqrt{ab},1,0,q,1,\alpha \right) \\ 
+\left\vert f^{\prime }\left( b\right) \right\vert ^{q}C_{4}\left( \sqrt{ab}%
,1,0,q,1,\alpha \right)%
\end{array}%
\right] ^{\frac{1}{q}}\right\} \text{.}
\end{eqnarray*}
\end{remark}

\begin{remark}
If we take $\theta =1$, $m=1$, $\alpha =1$ in Corollary \ref{Corollary3} we
have the following midpoint-type inequality for GA-convex function, which is
the same with the inequality $\left( 13\right) $ of Corollary 3.5 in \cite%
{19},%
\begin{eqnarray*}
&&\left\vert f\left( \sqrt{ab}\right) -\frac{1}{\ln b-\ln a}\int_{a}^{b}%
\frac{f\left( x\right) }{x}dx\right\vert \\
&\leq &\ln \frac{b}{a}\left( \frac{1}{2}\right) ^{3-\frac{1}{q}}\left\{ a%
\left[ 
\begin{array}{c}
\left\vert f^{\prime }\left( \sqrt{ab}\right) \right\vert ^{q}C_{1}\left( 
\sqrt{ab},1,0,q,1,1\right) \\ 
+\left\vert f^{\prime }\left( a\right) \right\vert ^{q}C_{2}\left( \sqrt{ab}%
,1,0,q,1,1\right)%
\end{array}%
\right] ^{\frac{1}{q}}\right. \\
&&\left. +b\left[ 
\begin{array}{c}
\left\vert f^{\prime }\left( \sqrt{ab}\right) \right\vert ^{q}C_{3}\left( 
\sqrt{ab},1,0,q,1,1\right) \\ 
+\left\vert f^{\prime }\left( b\right) \right\vert ^{q}C_{4}\left( \sqrt{ab}%
,1,0,q,1,1\right)%
\end{array}%
\right] ^{\frac{1}{q}}\right\} \text{.}
\end{eqnarray*}
\end{remark}

\begin{corollary}
Under the assumptions of Theorem \ref{Theorem5} with$\ x=\sqrt{ab}$, $%
\lambda =1$ from the inequality $\left( \ref{2.1}\right) $ we get the
following trepezoid-type inequality for fractional integrals 
\begin{eqnarray*}
&&\frac{2^{\theta -1}}{\left( m\ln \frac{b}{a}\right) ^{\theta }}\left\vert
K_{f}\left( 1,\theta ,\left( \sqrt{ab}\right) ^{m},a^{m},b^{m}\right)
\right\vert \\
&=&\left\vert \frac{f(a^{m})+f(b^{m})}{2}-\frac{2^{\theta -1}\Gamma \left(
\theta +1\right) }{\left( m\ln \frac{b}{a}\right) ^{\theta }}\left[
J_{\left( \sqrt{ab}\right) ^{m}-}^{\theta }f(a^{m})+J_{\left( \sqrt{ab}%
\right) ^{m}+}^{\theta }f(b^{m})\right] \right\vert \\
&\leq &\frac{m\ln \frac{b}{a}}{4}\left( \frac{\theta }{\theta +1}\right) ^{1-%
\frac{1}{q}}\left\{ a^{m}\left[ 
\begin{array}{c}
\left\vert f^{\prime }\left( \left( \sqrt{ab}\right) ^{m}\right) \right\vert
^{q}C_{1}\left( \sqrt{ab},\theta ,1,q,m,\alpha \right) \\ 
+m\left\vert f^{\prime }\left( a\right) \right\vert ^{q}C_{2}\left( \sqrt{ab}%
,\theta ,1,q,m,\alpha \right)%
\end{array}%
\right] ^{\frac{1}{q}}\right. \\
&&\left. +b^{m}\left[ 
\begin{array}{c}
\left\vert f^{\prime }\left( \left( \sqrt{ab}\right) ^{m}\right) \right\vert
^{q}C_{3}\left( \sqrt{ab},\theta ,1,q,m,\alpha \right) \\ 
+m\left\vert f^{\prime }\left( b\right) \right\vert ^{q}C_{4}\left( \sqrt{ab}%
,\theta ,1,q,m,\alpha \right)%
\end{array}%
\right] ^{\frac{1}{q}}\right\} .
\end{eqnarray*}
\end{corollary}

\begin{corollary}
Let the assumptions of Theorem \ref{Theorem5} hold. If $\ \left\vert
f^{\prime }(u)\right\vert \leq M$ for all $u\in \left[ a^{m},b\right] $ and $%
\lambda =0,$ then from the inequality $\left( \ref{2.1}\right) $ we get the
following Ostrowski type inequality for fractional integrals%
\begin{eqnarray*}
&&\left\vert \left[ \left( \ln \frac{x}{a}\right) ^{\theta }+\left( \ln 
\frac{b}{x}\right) ^{\theta }\right] f(x^{m})-\frac{\Gamma \left( \theta
+1\right) }{m^{\theta }}\left[ J_{x^{m}-}^{\theta
}f(a^{m})+J_{x^{m}+}^{\theta }f(b^{m})\right] \right\vert \\
&\leq &\frac{mM}{\left( \theta +1\right) ^{1-\frac{1}{q}}}\left\{
a^{m}\left( \ln \frac{x}{a}\right) ^{\theta +1}\left( 
\begin{array}{c}
C_{1}\left( x,\theta ,0,q,m,\alpha \right) \\ 
+mC_{2}\left( x,\theta ,0,q,m,\alpha \right)%
\end{array}%
\right) ^{\frac{1}{q}}\right. \\
&&\left. +b^{m}\left( \ln \frac{b}{x}\right) ^{\theta +1}\left( 
\begin{array}{c}
C_{3}\left( x,\theta ,\lambda ,q,m,\alpha \right) \\ 
+mC_{4}\left( x,\theta ,\lambda ,q,m,\alpha \right)%
\end{array}%
\right) ^{\frac{1}{q}}\right\}
\end{eqnarray*}%
for each $x\in \left[ a,b\right] .$
\end{corollary}

\begin{theorem}
\label{Theorem6}Let $f:I\subseteq \left( 0,\infty \right) \rightarrow 
\mathbb{R}
$ be a differentiable function on $I^{\circ }$ such that $f^{\prime }\in
L[a^{m},b^{m}]$, where $a^{m},b\in I$ $^{\circ }$ with $a<b$ and $m\in
\left( 0,1\right] $. If $|f^{\prime }|^{q}$ is $\left( \alpha ,m\right) $%
-GA-convex on $[a^{m},b]$ for some fixed $q>1$, $x\in \lbrack a,b]$, $%
\lambda \in \left[ 0,1\right] $ and $\theta >0$ then the following
inequality for fractional integrals holds%
\begin{eqnarray}
&&\left\vert K_{f}\left( \lambda ,\theta ,x^{m},a^{m},b^{m}\right)
\right\vert \leq m^{\theta +1}R_{0}^{^{\frac{1}{p}}}\left( \theta ,\lambda
,p\right)  \notag \\
&&\times \left\{ a^{m}\left( \ln \frac{x}{a}\right) ^{\theta +1}\left( 
\begin{array}{c}
\left\vert f^{\prime }\left( x^{m}\right) \right\vert ^{q}R_{1}\left(
x,q,m,\alpha \right) \\ 
+m\left\vert f^{\prime }\left( a\right) \right\vert ^{q}R_{2}\left(
x,q,m,\alpha \right)%
\end{array}%
\right) ^{\frac{1}{q}}\right.  \notag \\
&&\left. +b^{m}\left( \ln \frac{b}{x}\right) ^{\theta +1}\left( 
\begin{array}{c}
\left\vert f^{\prime }\left( x^{m}\right) \right\vert ^{q}R_{3}\left(
x,q,m,\alpha \right) \\ 
+m\left\vert f^{\prime }\left( b\right) \right\vert ^{q}R_{4}\left(
x,q,m,\alpha \right)%
\end{array}%
\right) ^{\frac{1}{q}}\right\}  \label{2.6}
\end{eqnarray}%
where%
\begin{equation*}
R_{0}\left( \theta ,\lambda ,p\right) =\dint\limits_{0}^{1}\left\vert
t^{\theta }-\lambda \right\vert ^{p}dt
\end{equation*}%
\begin{equation*}
=\left\{ 
\begin{array}{ccc}
\frac{1}{\theta p+1} & , & \lambda =0 \\ 
\begin{array}{c}
\left\{ \frac{\lambda ^{\left( \theta p+1\right) /\theta }}{\theta }\beta
\left( \frac{1}{\theta },p+1\right) +\frac{\left( 1-\lambda \right) ^{p+1}}{%
\theta \left( p+1\right) }\right. \\ 
\left. \times 
\begin{array}{c}
_{2}F_{1}\left( 1-\frac{1}{\theta },1;p+2;1-\lambda \right)%
\end{array}%
\right\}%
\end{array}
& , & 0<\lambda <1 \\ 
\frac{1}{\theta }\beta \left( \frac{1}{\theta },p+1\right) & , & \lambda =1%
\end{array}%
\right. ,
\end{equation*}%
\begin{eqnarray*}
R_{1}\left( x,q,m,\alpha \right) &=&\dint\limits_{0}^{1}\left( \frac{x}{a}%
\right) ^{mqt}t^{\alpha }dt, \\
R_{2}\left( x,q,m,\alpha \right) &=&\dint\limits_{0}^{1}\left( \frac{x}{a}%
\right) ^{mqt}\left( 1-t^{\alpha }\right) dt, \\
R_{3}\left( x,q,m,\alpha \right) &=&\dint\limits_{0}^{1}\left( \frac{x}{b}%
\right) ^{mqt}t^{\alpha }dt, \\
R_{4}\left( x,q,m,\alpha \right) &=&\dint\limits_{0}^{1}\left( \frac{x}{b}%
\right) ^{mqt}\left( 1-t^{\alpha }\right) dt,
\end{eqnarray*}%
$%
\begin{array}{c}
_{2}F_{1}%
\end{array}%
$ is hypergeometrical function defined by%
\begin{eqnarray*}
_{2}F_{1}\left( a,b;c;z\right) &=&\frac{1}{\beta \left( b,c-b\right) }%
\int_{0}^{1}t^{b-1}\left( 1-t\right) ^{c-b-1}\left( 1-zt\right) ^{-a}dt,%
\text{ } \\
c &>&b>0,\left\vert z\right\vert <1\left( \text{see \cite{4}}\right) ,
\end{eqnarray*}%
$\beta $ is beta function defined by%
\begin{equation*}
\beta \left( x,y\right) =\frac{\Gamma \left( x\right) \Gamma \left( y\right) 
}{\Gamma \left( x+y\right) }=\int_{0}^{1}t^{x-1}\left( 1-t\right) ^{y-1}dt,%
\text{ }x,y>0,
\end{equation*}

and $\frac{1}{p}+\frac{1}{q}=1$.
\end{theorem}

\begin{proof}
From Lemma \ref{Lemma2}, property of the modulus and using the H\"{o}lder
inequality we have%
\begin{eqnarray}
&&\left\vert K_{f}\left( \lambda ,\theta ,x^{m},a^{m},b^{m}\right)
\right\vert \leq m^{\theta +1}\left( \dint\limits_{0}^{1}\left\vert
t^{\theta }-\lambda \right\vert ^{p}dt\right) ^{\frac{1}{p}}  \notag \\
&&\times \left\{ a^{m}\left( \ln \frac{x}{a}\right) ^{\theta +1}\left(
\dint\limits_{0}^{1}\left( \frac{x}{a}\right) ^{qmt}\left\vert f^{\prime
}\left( x^{mt}a^{m\left( 1-t\right) }\right) \right\vert ^{q}dt\right) ^{%
\frac{1}{q}}\right.  \notag \\
&&\left. +b^{m}\left( \ln \frac{b}{x}\right) ^{\theta +1}\left(
\dint\limits_{0}^{1}\left( \frac{x}{b}\right) ^{qmt}\left\vert f^{\prime
}\left( x^{mt}b^{m\left( 1-t\right) }\right) \right\vert ^{q}dt\right) ^{%
\frac{1}{q}}\right\} .  \label{2.7}
\end{eqnarray}

By a simple computation%
\begin{equation*}
R_{0}\left( \theta ,\lambda ,p\right) =\dint\limits_{0}^{1}\left\vert
t^{\theta }-\lambda \right\vert ^{p}dt
\end{equation*}%
\begin{equation}
=\left\{ 
\begin{array}{ccc}
\frac{1}{\theta p+1} & , & \lambda =0 \\ 
\begin{array}{c}
\left\{ \frac{\lambda ^{\left( \theta p+1\right) /\theta }}{\theta }\beta
\left( \frac{1}{\theta },p+1\right) +\frac{\left( 1-\lambda \right) ^{p+1}}{%
\theta \left( p+1\right) }\right. \\ 
\left. \times 
\begin{array}{c}
_{2}F_{1}\left( 1-\frac{1}{\theta },1;p+2;1-\lambda \right)%
\end{array}%
\right\}%
\end{array}
& , & 0<\lambda <1 \\ 
\frac{1}{\theta }\beta \left( \frac{1}{\theta },p+1\right) & , & \lambda =1%
\end{array}%
\right. ,  \label{2.8}
\end{equation}%
Since $|f^{\prime }|^{q}$ is $\left( \alpha ,m\right) $-GA-convex on $%
[a^{m},b],$ for all $t\in \left[ 0,1\right] $, if we use $\left( \ref{2.3}%
\right) $, $\left( \ref{2.4}\right) $ and $\left( \ref{2.8}\right) $ in $%
\left( \ref{2.7}\right) $, we obtain $\left( \ref{2.6}\right) $. This
completes the proof.
\end{proof}

\begin{corollary}
\bigskip Under the assumptions of Theorem \ref{Theorem6} with $x=\sqrt{ab}$, 
$\lambda =\frac{1}{3}$ from the inequality $\left( \ref{2.6}\right) $ we get
the following Simpson type inequality for fractional integrals%
\begin{eqnarray*}
&&\frac{2^{\theta -1}}{\left( m\ln \frac{b}{a}\right) ^{\theta }}\left\vert
K_{f}\left( \frac{1}{3},\theta ,\left( \sqrt{ab}\right)
^{m},a^{m},b^{m}\right) \right\vert =\left\vert \frac{1}{6}\left[
f(a^{m})+4f\left( \left( \sqrt{ab}\right) ^{m}\right) +f(b^{m})\right]
\right. \\
&&-\frac{2^{\theta -1}\Gamma \left( \theta +1\right) }{\left( m\ln \frac{b}{a%
}\right) ^{\theta }}\left. \left[ J_{\left( \sqrt{ab}\right) ^{m}-}^{\theta
}f(a^{m})+J_{\left( \sqrt{ab}\right) ^{m}+}^{\theta }f(b^{m})\right]
\right\vert \leq \frac{m\ln \frac{b}{a}}{4}R_{0}^{^{\frac{1}{p}}}\left(
\theta ,\frac{1}{3},p\right) \\
&&\times \left\{ a^{m}\left[ 
\begin{array}{c}
\left\vert f^{\prime }\left( \left( \sqrt{ab}\right) ^{m}\right) \right\vert
^{q}R_{1}\left( \sqrt{ab},q,m,\alpha \right) \\ 
+m\left\vert f^{\prime }\left( a\right) \right\vert ^{q}R_{2}\left( \sqrt{ab}%
,q,m,\alpha \right)%
\end{array}%
\right] ^{\frac{1}{q}}\right. \\
&&\left. +b^{m}\left[ 
\begin{array}{c}
\left\vert f^{\prime }\left( \left( \sqrt{ab}\right) ^{m}\right) \right\vert
^{q}R_{3}\left( \sqrt{ab},q,m,\alpha \right) \\ 
+m\left\vert f^{\prime }\left( b\right) \right\vert ^{q}R_{4}\left( \sqrt{ab}%
,q,m,\alpha \right)%
\end{array}%
\right] ^{\frac{1}{q}}\right\}
\end{eqnarray*}
\end{corollary}

\begin{corollary}
\label{Corollary7}Under the assumptions of Theorem \ref{Theorem6} with $x=%
\sqrt{ab}$,$\ \lambda =0$ from the inequality $\left( \ref{2.6}\right) $ we
get the following midpoint-type inequality for fractional integrals
\end{corollary}

\begin{eqnarray*}
&&\frac{2^{\theta -1}}{\left( m\ln \frac{b}{a}\right) ^{\theta }}\left\vert
K_{f}\left( 0,\theta ,\left( \sqrt{ab}\right) ^{m},a^{m},b^{m}\right)
\right\vert \\
&=&\left\vert f\left( \left( \sqrt{ab}\right) ^{m}\right) -\frac{2^{\theta
-1}\Gamma \left( \theta +1\right) }{\left( m\ln \frac{b}{a}\right) ^{\theta }%
}\left[ J_{\left( \sqrt{ab}\right) ^{m}-}^{\theta }f(a^{m})+J_{\left( \sqrt{%
ab}\right) ^{m}+}^{\theta }f(b^{m})\right] \right\vert \\
&\leq &\frac{m\ln \frac{b}{a}}{4}\left( \frac{1}{\theta p+1}\right) ^{^{%
\frac{1}{p}}}\left\{ a^{m}\left[ 
\begin{array}{c}
\left\vert f^{\prime }\left( \left( \sqrt{ab}\right) ^{m}\right) \right\vert
^{q}R_{1}\left( \sqrt{ab},q,m,\alpha \right) \\ 
+m\left\vert f^{\prime }\left( a\right) \right\vert ^{q}R_{2}\left( \sqrt{ab}%
,q,m,\alpha \right)%
\end{array}%
\right] ^{\frac{1}{q}}\right. \\
&&\left. +b^{m}\left[ 
\begin{array}{c}
\left\vert f^{\prime }\left( \left( \sqrt{ab}\right) ^{m}\right) \right\vert
^{q}R_{3}\left( \sqrt{ab},q,m,\alpha \right) \\ 
+m\left\vert f^{\prime }\left( b\right) \right\vert ^{q}R_{4}\left( \sqrt{ab}%
,q,m,\alpha \right)%
\end{array}%
\right] ^{\frac{1}{q}}\right\}
\end{eqnarray*}

\begin{remark}
If we take $\theta =1$, $m=1$, $p=\frac{q}{q-1}$ in Corollary \ref%
{Corollary7} we have the following midpoint-type inequality for $\alpha $%
-GA-convex function (or GA-s-convex function in the first sense), which is
the same with the inequality $\left( 17\right) $ of Theorem 3.7.b. in \cite%
{19}, 
\begin{eqnarray*}
&&\left\vert f\left( \sqrt{ab}\right) -\frac{1}{\ln b-\ln a}\int_{a}^{b}%
\frac{f\left( x\right) }{x}dx\right\vert \\
&\leq &\frac{\ln \frac{b}{a}}{4}\left( \frac{q-1}{2q-1}\right) ^{1-\frac{1}{q%
}}\left\{ a\left[ 
\begin{array}{c}
\left\vert f^{\prime }\left( \sqrt{ab}\right) \right\vert ^{q}R_{1}\left( 
\sqrt{ab},q,1,\alpha \right) \\ 
+\left\vert f^{\prime }\left( a\right) \right\vert ^{q}R_{2}\left( \sqrt{ab}%
,q,1,\alpha \right)%
\end{array}%
\right] ^{\frac{1}{q}}\right. \\
&&\left. +b\left[ 
\begin{array}{c}
\left\vert f^{\prime }\left( \sqrt{ab}\right) \right\vert ^{q}R_{3}\left( 
\sqrt{ab},q,1,\alpha \right) \\ 
+\left\vert f^{\prime }\left( b\right) \right\vert ^{q}R_{4}\left( \sqrt{ab}%
,q,1,\alpha \right)%
\end{array}%
\right] ^{\frac{1}{q}}\right\} \text{.}
\end{eqnarray*}
\end{remark}

\begin{remark}
If we take $\theta =1$, $m=1$, $\alpha =1$, $p=\frac{q}{q-1}$ in Corollary %
\ref{Corollary7} we have the following midpoint-type inequality for
GA-convex function, which is the same with the inequality $\left( 21\right) $
of Corollary 3.8 in \cite{19},%
\begin{eqnarray*}
&&\left\vert f\left( \sqrt{ab}\right) -\frac{1}{\ln b-\ln a}\int_{a}^{b}%
\frac{f\left( x\right) }{x}dx\right\vert \\
&\leq &\frac{\ln \frac{b}{a}}{4}\left( \frac{q-1}{2q-1}\right) ^{1-\frac{1}{q%
}}\left\{ a\left[ 
\begin{array}{c}
\left\vert f^{\prime }\left( \sqrt{ab}\right) \right\vert ^{q}R_{1}\left( 
\sqrt{ab},q,1,1\right) \\ 
+\left\vert f^{\prime }\left( a\right) \right\vert ^{q}R_{2}\left( \sqrt{ab}%
,q,1,1\right)%
\end{array}%
\right] ^{\frac{1}{q}}\right. \\
&&\left. +b\left[ 
\begin{array}{c}
\left\vert f^{\prime }\left( \sqrt{ab}\right) \right\vert ^{q}R_{3}\left( 
\sqrt{ab},q,1,1\right) \\ 
+\left\vert f^{\prime }\left( b\right) \right\vert ^{q}R_{4}\left( \sqrt{ab}%
,q,1,1\right)%
\end{array}%
\right] ^{\frac{1}{q}}\right\} \text{..}
\end{eqnarray*}
\end{remark}

\begin{corollary}
Under the assumptions of Theorem \ref{Theorem6} with$\ x=\sqrt{ab}$, $%
\lambda =1$ from the inequality $\left( \ref{2.6}\right) $ we get the
following trepezoid-type inequality for fractional integrals 
\begin{eqnarray*}
&&\frac{2^{\theta -1}}{\left( m\ln \frac{b}{a}\right) ^{\theta }}\left\vert
K_{f}\left( 1,\theta ,\left( \sqrt{ab}\right) ^{m},a^{m},b^{m}\right)
\right\vert \\
&=&\left\vert \frac{f(a^{m})+f(b^{m})}{2}-\frac{2^{\theta -1}\Gamma \left(
\theta +1\right) }{\left( m\ln \frac{b}{a}\right) ^{\theta }}\left[
J_{\left( \sqrt{ab}\right) ^{m}-}^{\theta }f(a^{m})+J_{\left( \sqrt{ab}%
\right) ^{m}+}^{\theta }f(b^{m})\right] \right\vert \\
&\leq &\frac{m\ln \frac{b}{a}}{4}\left( \frac{1}{\theta }\beta \left( \frac{1%
}{\theta },p+1\right) \right) ^{\frac{1}{p}}\left\{ a^{m}\left[ 
\begin{array}{c}
\left\vert f^{\prime }\left( \left( \sqrt{ab}\right) ^{m}\right) \right\vert
^{q}R_{1}\left( \sqrt{ab},q,m,\alpha \right) \\ 
+m\left\vert f^{\prime }\left( a\right) \right\vert ^{q}R_{2}\left( \sqrt{ab}%
,q,m,\alpha \right)%
\end{array}%
\right] ^{\frac{1}{q}}\right. \\
&&\left. +b^{m}\left( 
\begin{array}{c}
\left\vert f^{\prime }\left( x^{m}\right) \right\vert ^{q}R_{3}\left( \sqrt{%
ab},q,m,\alpha \right) \\ 
+m\left\vert f^{\prime }\left( b\right) \right\vert ^{q}R_{4}\left( \sqrt{ab}%
,q,m,\alpha \right)%
\end{array}%
\right) ^{\frac{1}{q}}\right\}
\end{eqnarray*}
\end{corollary}

\begin{corollary}
Let the assumptions of Theorem \ref{Theorem6} hold. If $\ \left\vert
f^{\prime }(u)\right\vert \leq M$ for all $u\in \left[ a^{m},b\right] $ and $%
\lambda =0,$ then from the inequality $\left( \ref{2.6}\right) $ we get the
following Ostrowski type inequality for fractional integrals%
\begin{eqnarray*}
&&\left\vert \left[ \left( \ln \frac{x}{a}\right) ^{\theta }+\left( \ln 
\frac{b}{x}\right) ^{\theta }\right] f(x^{m})-\frac{\Gamma \left( \theta
+1\right) }{m^{\theta }}\left[ J_{x^{m}-}^{\theta
}f(a^{m})+J_{x^{m}+}^{\theta }f(b^{m})\right] \right\vert \\
&\leq &\frac{mM}{\left( \theta p+1\right) ^{\frac{1}{p}}}\left\{ a^{m}\left(
\ln \frac{x}{a}\right) ^{\theta +1}\left( 
\begin{array}{c}
R_{1}\left( x,q,m,\alpha \right) \\ 
+mR_{2}\left( x,q,m,\alpha \right)%
\end{array}%
\right) ^{\frac{1}{q}}\right. \\
&&\left. +b^{m}\left( \ln \frac{b}{x}\right) ^{\theta +1}\left( 
\begin{array}{c}
R_{3}\left( x,q,m,\alpha \right) \\ 
+R_{4}\left( x,q,m,\alpha \right)%
\end{array}%
\right) ^{\frac{1}{q}}\right\}
\end{eqnarray*}%
for each $x\in \left[ a,b\right] $
\end{corollary}

\begin{theorem}
\bigskip \label{Theorem7}Let $f:I\subseteq \left( 0,\infty \right)
\rightarrow 
\mathbb{R}
$ be a differentiable function on $I^{\circ }$ such that $f^{\prime }\in
L[a^{m},b^{m}]$, where $a^{m},b\in I$ $^{\circ }$ with $a<b$ and $m\in
\left( 0,1\right] $. If $|f^{\prime }|^{q}$ is $\left( \alpha ,m\right) $%
-GA-convex on $[a^{m},b]$ for some fixed $q>1$, $x\in \lbrack a,b]$, $%
\lambda \in \left[ 0,1\right] $ and $\theta >0$ then the following
inequality for fractional integrals holds%
\begin{eqnarray}
&&\left\vert K_{f}\left( \lambda ,\theta ,x^{m},a^{m},b^{m}\right)
\right\vert \leq m^{\theta +1}  \notag \\
&&\times \left\{ a^{m}\left( \ln \frac{x}{a}\right) ^{\theta +1}T_{1}^{\frac{%
1}{p}}\left( x,\theta ,\lambda ,p,m\right) \left( \frac{\left\vert f^{\prime
}\left( x^{m}\right) \right\vert ^{q}+m\alpha \left\vert f^{\prime }\left(
a\right) \right\vert ^{q}}{\alpha +1}\right) ^{\frac{1}{q}}\right.  \notag \\
&&\left. +b^{m}\left( \ln \frac{b}{x}\right) ^{\theta +1}T_{2}^{\frac{1}{p}%
}\left( x,\theta ,\lambda ,p,m\right) \left( \frac{\left\vert f^{\prime
}\left( x^{m}\right) \right\vert ^{q}+m\alpha \left\vert f^{\prime }\left(
b\right) \right\vert ^{q}}{\alpha +1}\right) ^{\frac{1}{q}}\right\}
\label{2.9}
\end{eqnarray}%
where%
\begin{eqnarray*}
T_{1}\left( x,\theta ,\lambda ,p,m\right) &=&\dint\limits_{0}^{1}\left\vert
t^{\theta }-\lambda \right\vert ^{p}\left( \frac{x}{a}\right) ^{mpt}dt, \\
T_{2}\left( x,\theta ,\lambda ,p,m\right) &=&\dint\limits_{0}^{1}\left\vert
t^{\theta }-\lambda \right\vert ^{p}\left( \frac{x}{b}\right) ^{mpt}dt,
\end{eqnarray*}%
and $\frac{1}{p}+\frac{1}{q}=1$.
\end{theorem}

\begin{proof}
Since $|f^{\prime }|^{q}$ is $\left( \alpha ,m\right) $-GA-convex on $%
[a^{m},b],$ for all $t\in \left[ 0,1\right] $, if we use $\left( \ref{2.3}%
\right) $, $\left( \ref{2.4}\right) $%
\begin{eqnarray}
\dint\limits_{0}^{1}\left\vert f^{\prime }\left( x^{mt}a^{m\left( 1-t\right)
}\right) \right\vert ^{q}dt &\leq &\dint\limits_{0}^{1}t^{\alpha }\left\vert
f^{\prime }\left( x^{m}\right) \right\vert ^{q}+m\left( 1-t^{\alpha }\right)
\left\vert f^{\prime }\left( a\right) \right\vert ^{q}dt  \notag \\
&=&\frac{\left\vert f^{\prime }\left( x^{m}\right) \right\vert ^{q}+m\alpha
\left\vert f^{\prime }\left( a\right) \right\vert ^{q}}{\alpha +1},
\label{2.10}
\end{eqnarray}%
\begin{eqnarray}
\dint\limits_{0}^{1}\left\vert f^{\prime }\left( x^{mt}b^{m\left( 1-t\right)
}\right) \right\vert ^{q}dt &\leq &\dint\limits_{0}^{1}t^{\alpha }\left\vert
f^{\prime }\left( x^{m}\right) \right\vert ^{q}+m\left( 1-t^{\alpha }\right)
\left\vert f^{\prime }\left( b\right) \right\vert ^{q}dt  \notag \\
&=&\frac{\left\vert f^{\prime }\left( x^{m}\right) \right\vert ^{q}+m\alpha
\left\vert f^{\prime }\left( b\right) \right\vert ^{q}}{\alpha +1}.
\label{2.11}
\end{eqnarray}

From Lemma \ref{Lemma2}, property of the modulus, $\left( \ref{2.10}\right) $%
, $\left( \ref{2.11}\right) $ and using the H\"{o}lder inequality, we have%
\begin{eqnarray*}
&&\left\vert K_{f}\left( \lambda ,\theta ,x^{m},a^{m},b^{m}\right)
\right\vert \leq m^{\theta +1} \\
&&\times \left\{ a^{m}\left( \ln \frac{x}{a}\right) ^{\theta +1}\left(
\dint\limits_{0}^{1}\left\vert t^{\theta }-\lambda \right\vert ^{p}\left( 
\frac{x}{a}\right) ^{mpt}dt\right) ^{\frac{1}{p}}\left(
\dint\limits_{0}^{1}\left\vert f^{\prime }\left( x^{mt}a^{m\left( 1-t\right)
}\right) \right\vert ^{q}dt\right) ^{\frac{1}{q}}\right. \\
&&\left. +b^{m}\left( \ln \frac{b}{x}\right) ^{\theta +1}\left(
\dint\limits_{0}^{1}\left\vert t^{\theta }-\lambda \right\vert ^{p}\left( 
\frac{x}{b}\right) ^{mpt}dt\right) ^{\frac{1}{p}}\left(
\dint\limits_{0}^{1}\left\vert f^{\prime }\left( x^{mt}b^{m\left( 1-t\right)
}\right) \right\vert ^{q}dt\right) ^{\frac{1}{q}}\right\} \\
&\leq &m^{\theta +1}\left\{ a^{m}\left( \ln \frac{x}{a}\right) ^{\theta
+1}\left( \dint\limits_{0}^{1}\left\vert t^{\theta }-\lambda \right\vert
^{p}\left( \frac{x}{a}\right) ^{mpt}dt\right) ^{\frac{1}{p}}\left( \frac{%
\left\vert f^{\prime }\left( x^{m}\right) \right\vert ^{q}+m\alpha
\left\vert f^{\prime }\left( a\right) \right\vert ^{q}}{\alpha +1}\right) ^{%
\frac{1}{q}}\right. \\
&&\left. +b^{m}\left( \ln \frac{b}{x}\right) ^{\theta +1}\left(
\dint\limits_{0}^{1}\left\vert t^{\theta }-\lambda \right\vert ^{p}\left( 
\frac{x}{b}\right) ^{mpt}dt\right) ^{\frac{1}{p}}\left( \frac{\left\vert
f^{\prime }\left( x^{m}\right) \right\vert ^{q}+m\alpha \left\vert f^{\prime
}\left( b\right) \right\vert ^{q}}{\alpha +1}\right) ^{\frac{1}{q}}\right\}
\end{eqnarray*}

This completes the proof.
\end{proof}

\begin{corollary}
\bigskip Under the assumptions of Theorem \ref{Theorem7} with $x=\sqrt{ab}$, 
$\lambda =\frac{1}{3}$ from the inequality $\left( \ref{2.9}\right) $ we get
the following Simpson type inequality for fractional integrals%
\begin{eqnarray*}
&&\frac{2^{\theta -1}}{\left( m\ln \frac{b}{a}\right) ^{\theta }}\left\vert
K_{f}\left( \frac{1}{3},\theta ,\left( \sqrt{ab}\right)
^{m},a^{m},b^{m}\right) \right\vert =\left\vert \frac{1}{6}\left[
f(a^{m})+4f\left( \left( \sqrt{ab}\right) ^{m}\right) +f(b^{m})\right]
\right. \\
&&-\frac{2^{\theta -1}\Gamma \left( \theta +1\right) }{\left( m\ln \frac{b}{a%
}\right) ^{\theta }}\left. \left[ J_{\left( \sqrt{ab}\right) ^{m}-}^{\theta
}f(a^{m})+J_{\left( \sqrt{ab}\right) ^{m}+}^{\theta }f(b^{m})\right]
\right\vert \leq \frac{m\ln \frac{b}{a}}{4} \\
&&\times \left\{ a^{m}T_{1}^{\frac{1}{p}}\left( \sqrt{ab},\theta ,\frac{1}{3}%
,p,m\right) \left( \frac{\left\vert f^{\prime }\left( \left( \sqrt{ab}%
\right) ^{m}\right) \right\vert ^{q}+m\alpha \left\vert f^{\prime }\left(
a\right) \right\vert ^{q}}{\alpha +1}\right) ^{\frac{1}{q}}\right. \\
&&\left. +b^{m}T_{2}^{\frac{1}{p}}\left( \sqrt{ab},\theta ,\frac{1}{3}%
,p,m\right) \left( \frac{\left\vert f^{\prime }\left( \left( \sqrt{ab}%
\right) ^{m}\right) \right\vert ^{q}+m\alpha \left\vert f^{\prime }\left(
b\right) \right\vert ^{q}}{\alpha +1}\right) ^{\frac{1}{q}}\right\}
\end{eqnarray*}
\end{corollary}

\begin{corollary}
Under the assumptions of Theorem \ref{Theorem7} with $x=\sqrt{ab}$,$\
\lambda =0$ from the inequality $\left( \ref{2.9}\right) $ we get the
following midpoint-type inequality for fractional integrals
\end{corollary}

\begin{eqnarray*}
&&\frac{2^{\theta -1}}{\left( m\ln \frac{b}{a}\right) ^{\theta }}\left\vert
K_{f}\left( 0,\theta ,\left( \sqrt{ab}\right) ^{m},a^{m},b^{m}\right)
\right\vert \\
&=&\left\vert f\left( \left( \sqrt{ab}\right) ^{m}\right) -\frac{2^{\theta
-1}\Gamma \left( \theta +1\right) }{\left( m\ln \frac{b}{a}\right) ^{\theta }%
}\left[ J_{\left( \sqrt{ab}\right) ^{m}-}^{\theta }f(a^{m})+J_{\left( \sqrt{%
ab}\right) ^{m}+}^{\theta }f(b^{m})\right] \right\vert \\
&\leq &\frac{m\ln \frac{b}{a}}{4}\left\{ a^{m}T_{1}^{\frac{1}{p}}\left( 
\sqrt{ab},\theta ,0,p,m\right) \left( \frac{\left\vert f^{\prime }\left(
\left( \sqrt{ab}\right) ^{m}\right) \right\vert ^{q}+m\alpha \left\vert
f^{\prime }\left( a\right) \right\vert ^{q}}{\alpha +1}\right) ^{\frac{1}{q}%
}\right. \\
&&\left. +b^{m}T_{2}^{\frac{1}{p}}\left( \sqrt{ab},\theta ,0,p,m\right)
\left( \frac{\left\vert f^{\prime }\left( \left( \sqrt{ab}\right)
^{m}\right) \right\vert ^{q}+m\alpha \left\vert f^{\prime }\left( b\right)
\right\vert ^{q}}{\alpha +1}\right) ^{\frac{1}{q}}\right\}
\end{eqnarray*}

\begin{corollary}
Under the assumptions of Theorem \ref{Theorem7} with$\ x=\sqrt{ab}$, $%
\lambda =1$ from the inequality $\left( \ref{2.9}\right) $ we get the
following trepezoid-type inequality for fractional integrals 
\begin{eqnarray*}
&&\frac{2^{\theta -1}}{\left( m\ln \frac{b}{a}\right) ^{\theta }}\left\vert
K_{f}\left( 1,\theta ,\left( \sqrt{ab}\right) ^{m},a^{m},b^{m}\right)
\right\vert \\
&=&\left\vert \frac{f(a^{m})+f(b^{m})}{2}-\frac{2^{\theta -1}\Gamma \left(
\theta +1\right) }{\left( m\ln \frac{b}{a}\right) ^{\theta }}\left[
J_{\left( \sqrt{ab}\right) ^{m}-}^{\theta }f(a^{m})+J_{\left( \sqrt{ab}%
\right) ^{m}+}^{\theta }f(b^{m})\right] \right\vert \\
&\leq &\frac{m\ln \frac{b}{a}}{4}\left\{ a^{m}T_{1}^{\frac{1}{p}}\left( 
\sqrt{ab},\theta ,1,p,m\right) \left( \frac{\left\vert f^{\prime }\left(
\left( \sqrt{ab}\right) ^{m}\right) \right\vert ^{q}+m\alpha \left\vert
f^{\prime }\left( a\right) \right\vert ^{q}}{\alpha +1}\right) ^{\frac{1}{q}%
}\right. \\
&&\left. +b^{m}T_{2}^{\frac{1}{p}}\left( \sqrt{ab},\theta ,1,p,m\right)
\left( \frac{\left\vert f^{\prime }\left( \left( \sqrt{ab}\right)
^{m}\right) \right\vert ^{q}+m\alpha \left\vert f^{\prime }\left( b\right)
\right\vert ^{q}}{\alpha +1}\right) ^{\frac{1}{q}}\right\}
\end{eqnarray*}
\end{corollary}

\begin{corollary}
Let the assumptions of Theorem \ref{Theorem7} hold. If $\ \left\vert
f^{\prime }(u)\right\vert \leq M$ for all $u\in \left[ a^{m},b\right] $ and $%
\lambda =0,$ then from the inequality $\left( \ref{2.9}\right) $ we get the
following Ostrowski type inequality for fractional integrals%
\begin{eqnarray*}
&&\left\vert \left[ \left( \ln \frac{x}{a}\right) ^{\theta }+\left( \ln 
\frac{b}{x}\right) ^{\theta }\right] f(x^{m})-\frac{\Gamma \left( \theta
+1\right) }{m^{\theta }}\left[ J_{x^{m}-}^{\theta
}f(a^{m})+J_{x^{m}+}^{\theta }f(b^{m})\right] \right\vert \\
&\leq &mM\left( \frac{1+m\alpha }{\alpha +1}\right) ^{\frac{1}{q}} \\
&&\times \left\{ a^{m}\left( \ln \frac{x}{a}\right) ^{\theta +1}T_{1}^{\frac{%
1}{p}}\left( x,\theta ,0,p,m\right) +b^{m}\left( \ln \frac{b}{x}\right)
^{\theta +1}T_{2}^{\frac{1}{p}}\left( x,\theta ,0,p,m\right) \right\}
\end{eqnarray*}%
for each $x\in \left[ a,b\right] $
\end{corollary}

\begin{theorem}
\bigskip \label{Theorem8}Let $f:I\subseteq \left( 0,\infty \right)
\rightarrow 
\mathbb{R}
$ be a differentiable function on $I^{\circ }$ such that $f^{\prime }\in
L[a^{m},b^{m}]$, where $a^{m},b\in I$ $^{\circ }$ with $a<b$ and $m\in
\left( 0,1\right] $. If $|f^{\prime }|^{q}$ is $\left( \alpha ,m\right) $%
-GA-convex on $[a^{m},b]$ for some fixed $q>1$, $x\in \lbrack a,b]$, $%
\lambda \in \left[ 0,1\right] $ and $\theta >0$ then the following
inequality for fractional integrals holds%
\begin{eqnarray}
&&\left\vert K_{f}\left( \lambda ,\theta ,x^{m},a^{m},b^{m}\right)
\right\vert \leq m^{\theta +1}  \notag \\
&&\times \left\{ a^{m}\left( \ln \frac{x}{a}\right) ^{\theta +1}V_{3}^{\frac{%
1}{p}}\left[ 
\begin{array}{c}
V_{1}\left( \theta ,\lambda ,\alpha ,q\right) \left\vert f^{\prime }\left(
x^{m}\right) \right\vert ^{q} \\ 
+mV_{2}\left( \theta ,\lambda ,\alpha ,q\right) \left\vert f^{\prime }\left(
a\right) \right\vert ^{q}%
\end{array}%
\right] ^{\frac{1}{q}}\right.   \notag \\
&&\left. +b^{m}\left( \ln \frac{b}{x}\right) ^{\theta +1}V_{4}^{\frac{1}{p}}%
\left[ 
\begin{array}{c}
V_{1}\left( \theta ,\lambda ,\alpha ,q\right) \left\vert f^{\prime }\left(
x^{m}\right) \right\vert ^{q} \\ 
+mV_{2}\left( \theta ,\lambda ,\alpha ,q\right) \left\vert f^{\prime }\left(
b\right) \right\vert ^{q}%
\end{array}%
\right] ^{\frac{1}{q}}\right\}   \label{2.12}
\end{eqnarray}%
where%
\begin{equation*}
V_{1}\left( \theta ,\lambda ,\alpha ,q\right)
=\dint\limits_{0}^{1}\left\vert t^{\theta }-\lambda \right\vert
^{q}t^{\alpha }dt
\end{equation*}%
\begin{equation}
=\left\{ 
\begin{array}{ccc}
\frac{1}{\theta q+\alpha +1} & , & \lambda =0 \\ 
\begin{array}{c}
\left\{ \frac{\lambda ^{\left( \theta q+\alpha +1\right) /\theta }}{\theta }%
\beta \left( \frac{\alpha +1}{\theta },q+1\right) +\frac{\left( 1-\lambda
\right) ^{q+1}}{\theta \left( q+1\right) }\right.  \\ 
\left. \times 
\begin{array}{c}
_{2}F_{1}\left( 1-\frac{\alpha +1}{\theta },1;q+2;1-\lambda \right) 
\end{array}%
\right\} 
\end{array}
& , & 0<\lambda <1 \\ 
\frac{1}{\theta }\beta \left( \frac{\alpha +1}{\theta },q+1\right)  & , & 
\lambda =1%
\end{array}%
\right.   \label{2.13}
\end{equation}%
\begin{equation*}
V_{2}\left( \theta ,\lambda ,\alpha ,q\right)
=\dint\limits_{0}^{1}\left\vert t^{\theta }-\lambda \right\vert ^{q}\left(
1-t^{\alpha }\right) dt
\end{equation*}%
\begin{equation}
=\left\{ 
\begin{array}{ccc}
\frac{1}{\theta q+1}-\frac{1}{\theta q+\alpha +1} & , & \lambda =0 \\ 
\begin{array}{c}
\left\{ \frac{\lambda ^{\left( \theta q+1\right) /\theta }}{\theta }\beta
\left( \frac{1}{\theta },q+1\right) -\frac{\lambda ^{\left( \theta q+\alpha
+1\right) /\theta }}{\theta }\beta \left( \frac{\alpha +1}{\theta }%
,q+1\right) \right.  \\ 
\left. +\frac{\left( 1-\lambda \right) ^{q+1}}{\theta \left( q+1\right) }%
\left( 
\begin{array}{c}
\begin{array}{c}
_{2}F_{1}\left( 1-\frac{1}{\theta },1;q+2;1-\lambda \right) 
\end{array}
\\ 
-_{2}F_{1}\left( 1-\frac{\alpha +1}{\theta },1;q+2;1-\lambda \right) 
\end{array}%
\right) \right\} 
\end{array}
& , & 0<\lambda <1 \\ 
\frac{1}{\theta }\beta \left( \frac{1}{\theta },q+1\right) -\frac{1}{\theta }%
\beta \left( \frac{\alpha +1}{\theta },q+1\right)  & , & \lambda =1%
\end{array}%
\right.   \label{2.14}
\end{equation}%
\begin{equation}
V_{3}=\dint\limits_{0}^{1}\left( \frac{x}{a}\right) ^{pmt}dt=\left\{ 
\begin{array}{ccc}
\frac{\left( \frac{x}{a}\right) ^{mp}-1}{\ln \left( \frac{x}{a}\right) ^{mp}}
& , & x\neq a \\ 
1 & , & \text{otherwise}%
\end{array}%
\right.   \label{2.15}
\end{equation}%
\begin{equation}
V_{4}=\dint\limits_{0}^{1}\left( \frac{x}{b}\right) ^{pmt}dt=\left\{ 
\begin{array}{ccc}
\frac{\left( \frac{x}{b}\right) ^{mp}-1}{\ln \left( \frac{x}{b}\right) ^{mp}}
& , & x\neq b \\ 
1 & , & \text{otherwise}%
\end{array}%
\right.   \label{2.16}
\end{equation}

and $\frac{1}{p}+\frac{1}{q}=1$.
\end{theorem}

\begin{proof}
From Lemma \ref{Lemma2}, property of the modulus, the H\"{o}lder inequality
and by using $\left( \ref{2.3}\right) $, $\left( \ref{2.4}\right) $, $\left( %
\ref{2.15}\right) $ and $\left( \ref{2.16}\right) $ we have%
\begin{equation*}
\left\vert K_{f}\left( \lambda ,\theta ,x^{m},a^{m},b^{m}\right) \right\vert
\leq m^{\theta +1}
\end{equation*}%
\begin{eqnarray*}
&&\times \left\{ a^{m}\left( \ln \frac{x}{a}\right) ^{\theta +1}\left(
\dint\limits_{0}^{1}\left( \frac{x}{a}\right) ^{pmt}dt\right) ^{\frac{1}{p}%
}\left( \dint\limits_{0}^{1}\left\vert t^{\theta }-\lambda \right\vert
^{q}\left\vert f^{\prime }\left( x^{mt}a^{m\left( 1-t\right) }\right)
\right\vert ^{q}dt\right) ^{\frac{1}{q}}\right.  \\
&&\left. +b^{m}\left( \ln \frac{b}{x}\right) ^{\theta +1}\left(
\dint\limits_{0}^{1}\left( \frac{x}{b}\right) ^{pmt}dt\right) ^{\frac{1}{p}%
}\left( \dint\limits_{0}^{1}\left\vert t^{\theta }-\lambda \right\vert
^{q}\left\vert f^{\prime }\left( x^{mt}b^{m\left( 1-t\right) }\right)
\right\vert ^{q}dt\right) ^{\frac{1}{q}}\right\} 
\end{eqnarray*}%
\begin{eqnarray}
&\leq &m^{\theta +1}\left\{ a^{m}\left( \ln \frac{x}{a}\right) ^{\theta
+1}V_{3}^{\frac{1}{p}}\right.   \notag \\
&&\times \left( \dint\limits_{0}^{1}\left\vert t^{\theta }-\lambda
\right\vert ^{q}\left[ 
\begin{array}{c}
t^{\alpha }\left\vert f^{\prime }\left( x^{m}\right) \right\vert ^{q} \\ 
+m\left( 1-t^{\alpha }\right) \left\vert f^{\prime }\left( a\right)
\right\vert ^{q}%
\end{array}%
\right] dt\right) ^{\frac{1}{q}}  \notag \\
&&+b^{m}\left( \ln \frac{b}{x}\right) ^{\theta +1}V_{4}^{\frac{1}{p}}  \notag
\\
&&\left. \times \left( \dint\limits_{0}^{1}\left\vert t^{\theta }-\lambda
\right\vert ^{q}\left[ 
\begin{array}{c}
t^{\alpha }\left\vert f^{\prime }\left( x^{m}\right) \right\vert ^{q} \\ 
+m\left( 1-t^{\alpha }\right) \left\vert f^{\prime }\left( b\right)
\right\vert ^{q}%
\end{array}%
\right] dt\right) ^{\frac{1}{q}}\right\}   \label{2.17}
\end{eqnarray}

By a simple computation we verify $\left( \ref{2.13}\right) $ and $\left( %
\ref{2.14}\right) $. If we use $\left( \ref{2.13}\right) $, $\left( \ref%
{2.14}\right) $, $\left( \ref{2.15}\right) $ and $\left( \ref{2.16}\right) $
in $\left( \ref{2.17}\right) $ we obtain $\left( \ref{2.12}\right) $. This
completes the proof.
\end{proof}

\begin{corollary}
\bigskip Under the assumptions of Theorem \ref{Theorem8} with $x=\sqrt{ab}$, 
$\lambda =\frac{1}{3}$ from the inequality $\left( \ref{2.12}\right) $ we
get the following Simpson type inequality for fractional integrals%
\begin{eqnarray*}
&&\frac{2^{\theta -1}}{\left( m\ln \frac{b}{a}\right) ^{\theta }}\left\vert
K_{f}\left( \frac{1}{3},\theta ,\left( \sqrt{ab}\right)
^{m},a^{m},b^{m}\right) \right\vert =\left\vert \frac{1}{6}\left[
f(a^{m})+4f\left( \left( \sqrt{ab}\right) ^{m}\right) +f(b^{m})\right]
\right. \\
&&-\frac{2^{\theta -1}\Gamma \left( \theta +1\right) }{\left( m\ln \frac{b}{a%
}\right) ^{\theta }}\left. \left[ J_{\left( \sqrt{ab}\right) ^{m}-}^{\theta
}f(a^{m})+J_{\left( \sqrt{ab}\right) ^{m}+}^{\theta }f(b^{m})\right]
\right\vert \leq \frac{m\ln \frac{b}{a}}{4} \\
&&\times \left\{ a^{m}\left( \frac{\left( \frac{b}{a}\right) ^{\frac{mp}{2}%
}-1}{\ln \left( \frac{b}{a}\right) ^{\frac{mp}{2}}}\right) ^{\frac{1}{p}}%
\left[ 
\begin{array}{c}
V_{1}\left( \theta ,\frac{1}{3},\alpha ,q\right) \left\vert f^{\prime
}\left( x^{m}\right) \right\vert ^{q} \\ 
+mV_{2}\left( \theta ,\frac{1}{3},\alpha ,q\right) \left\vert f^{\prime
}\left( a\right) \right\vert ^{q}%
\end{array}%
\right] ^{\frac{1}{q}}\right. \\
&&\left. +b^{m}\left( \frac{\left( \frac{a}{b}\right) ^{\frac{mp}{2}}-1}{\ln
\left( \frac{a}{b}\right) ^{\frac{mp}{2}}}\right) ^{\frac{1}{p}}\left[ 
\begin{array}{c}
V_{1}\left( \theta ,\frac{1}{3},\alpha ,q\right) \left\vert f^{\prime
}\left( x^{m}\right) \right\vert ^{q} \\ 
+mV_{2}\left( \theta ,\frac{1}{3},\alpha ,q\right) \left\vert f^{\prime
}\left( b\right) \right\vert ^{q}%
\end{array}%
\right] ^{\frac{1}{q}}\right\}
\end{eqnarray*}
\end{corollary}

\begin{corollary}
Under the assumptions of Theorem \ref{Theorem8} with $x=\sqrt{ab}$,$\
\lambda =0$ from the inequality $\left( \ref{2.12}\right) $ we get the
following midpoint-type inequality for fractional integrals
\end{corollary}

\begin{eqnarray*}
&&\frac{2^{\theta -1}}{\left( m\ln \frac{b}{a}\right) ^{\theta }}\left\vert
K_{f}\left( 0,\theta ,\left( \sqrt{ab}\right) ^{m},a^{m},b^{m}\right)
\right\vert \\
&=&\left\vert f\left( \left( \sqrt{ab}\right) ^{m}\right) -\frac{2^{\theta
-1}\Gamma \left( \theta +1\right) }{\left( m\ln \frac{b}{a}\right) ^{\theta }%
}\left[ J_{\left( \sqrt{ab}\right) ^{m}-}^{\theta }f(a^{m})+J_{\left( \sqrt{%
ab}\right) ^{m}+}^{\theta }f(b^{m})\right] \right\vert \\
&\leq &\frac{m\ln \frac{b}{a}}{4}\left\{ a^{m}\left( \frac{\left( \frac{b}{a}%
\right) ^{\frac{mp}{2}}-1}{\ln \left( \frac{b}{a}\right) ^{\frac{mp}{2}}}%
\right) ^{\frac{1}{p}}\left[ 
\begin{array}{c}
\frac{1}{\theta q+\alpha +1}\left\vert f^{\prime }\left( x^{m}\right)
\right\vert ^{q} \\ 
+\left( \frac{m}{\theta q+1}-\frac{m}{\theta q+\alpha +1}\right) \left\vert
f^{\prime }\left( a\right) \right\vert ^{q}%
\end{array}%
\right] ^{\frac{1}{q}}\right. \\
&&\left. +b^{m}\left( \frac{\left( \frac{a}{b}\right) ^{\frac{mp}{2}}-1}{\ln
\left( \frac{a}{b}\right) ^{\frac{mp}{2}}}\right) ^{\frac{1}{p}}\left[ 
\begin{array}{c}
\frac{1}{\theta q+\alpha +1}\left\vert f^{\prime }\left( x^{m}\right)
\right\vert ^{q} \\ 
+\left( \frac{m}{\theta q+1}-\frac{m}{\theta q+\alpha +1}\right) \left\vert
f^{\prime }\left( b\right) \right\vert ^{q}%
\end{array}%
\right] ^{\frac{1}{q}}\right\}
\end{eqnarray*}

\begin{corollary}
Under the assumptions of Theorem \ref{Theorem8} with$\ x=\sqrt{ab}$, $%
\lambda =1$ from the inequality $\left( \ref{2.12}\right) $ we get the
following trepezoid-type inequality for fractional integrals 
\begin{eqnarray*}
&&\frac{2^{\theta -1}}{\left( m\ln \frac{b}{a}\right) ^{\theta }}\left\vert
K_{f}\left( 1,\theta ,\left( \sqrt{ab}\right) ^{m},a^{m},b^{m}\right)
\right\vert \\
&=&\left\vert \frac{f(a^{m})+f(b^{m})}{2}-\frac{2^{\theta -1}\Gamma \left(
\theta +1\right) }{\left( m\ln \frac{b}{a}\right) ^{\theta }}\left[
J_{\left( \sqrt{ab}\right) ^{m}-}^{\theta }f(a^{m})+J_{\left( \sqrt{ab}%
\right) ^{m}+}^{\theta }f(b^{m})\right] \right\vert \\
&\leq &\frac{m\ln \frac{b}{a}}{4}\left\{ a^{m}\left( \frac{\left( \frac{b}{a}%
\right) ^{\frac{mp}{2}}-1}{\ln \left( \frac{b}{a}\right) ^{\frac{mp}{2}}}%
\right) ^{\frac{1}{p}}\left[ 
\begin{array}{c}
\frac{1}{\theta }\beta \left( \frac{\alpha +1}{\theta },q+1\right)
\left\vert f^{\prime }\left( x^{m}\right) \right\vert ^{q} \\ 
+\left( \frac{1}{\theta }\beta \left( \frac{1}{\theta },q+1\right) -\frac{1}{%
\theta }\beta \left( \frac{\alpha +1}{\theta },q+1\right) \right) \left\vert
f^{\prime }\left( a\right) \right\vert ^{q}%
\end{array}%
\right] ^{\frac{1}{q}}\right. \\
&&\left. +b^{m}\left( \frac{\left( \frac{a}{b}\right) ^{\frac{mp}{2}}-1}{\ln
\left( \frac{a}{b}\right) ^{\frac{mp}{2}}}\right) ^{\frac{1}{p}}\left[ 
\begin{array}{c}
\frac{1}{\theta }\beta \left( \frac{\alpha +1}{\theta },q+1\right)
\left\vert f^{\prime }\left( x^{m}\right) \right\vert ^{q} \\ 
+\left( \frac{1}{\theta }\beta \left( \frac{1}{\theta },q+1\right) -\frac{1}{%
\theta }\beta \left( \frac{\alpha +1}{\theta },q+1\right) \right) \left\vert
f^{\prime }\left( b\right) \right\vert ^{q}%
\end{array}%
\right] ^{\frac{1}{q}}\right\}
\end{eqnarray*}
\end{corollary}

\begin{corollary}
Let the assumptions of Theorem \ref{Theorem7} hold. If $\ \left\vert
f^{\prime }(u)\right\vert \leq M$ for all $u\in \left[ a^{m},b\right] $ and $%
\lambda =0,$ then from the inequality $\left( \ref{2.12}\right) $ we get the
following Ostrowski type inequality for fractional integrals%
\begin{eqnarray*}
&&\left\vert \left[ \left( \ln \frac{x}{a}\right) ^{\theta }+\left( \ln 
\frac{b}{x}\right) ^{\theta }\right] f(x^{m})-\frac{\Gamma \left( \theta
+1\right) }{m^{\theta }}\left[ J_{x^{m}-}^{\theta
}f(a^{m})+J_{x^{m}+}^{\theta }f(b^{m})\right] \right\vert \\
&\leq &mM\left[ 
\begin{array}{c}
\frac{1}{\theta q+\alpha +1} \\ 
+\left( \frac{m}{\theta q+1}-\frac{m}{\theta q+\alpha +1}\right)%
\end{array}%
\right] ^{\frac{1}{q}}\left\{ a^{m}\left( \ln \frac{x}{a}\right) ^{\theta
+1}\left( \frac{\left( \frac{x}{a}\right) ^{mp}-1}{\ln \left( \frac{x}{a}%
\right) ^{mp}}\right) ^{\frac{1}{p}}\right. \\
&&\left. +b^{m}\left( \ln \frac{b}{x}\right) ^{\theta +1}\left( \frac{\left( 
\frac{x}{b}\right) ^{mp}-1}{\ln \left( \frac{x}{b}\right) ^{mp}}\right) ^{%
\frac{1}{p}}\right\}
\end{eqnarray*}%
for each $x\in \left[ a,b\right] $
\end{corollary}


\begin{thebibliography}{99}
\bibitem{1} M. Alomaria, M. Darus, S.S. Dragomir, P. Cerone, Ostrowski type
inequalities for functions whose derivatives are $s$-convex in the second
sense, Applied Mathematics Letters 23 (2010) 1071--1076.

\bibitem{2} M. Avci, H. Kavurmaci and M.E. Ozdemir, New inequalities of
Hermite-Hadamard type via $s$-convex functions in the second sense with
applications, Appl. Math. Comput., 217 (2011) 5171--5176.

\bibitem{3} Z. Dahmani, On Minkowski and Hermite-Hadamard integral
inequalities via fractional via fractional integration, Ann. Funct. Anal. 1
(1) (2010), 51-58

\bibitem{4} A. A. Kilbas, H. M. Srivastava, J. J. Trujillo, Theory and
Applications of Fractional Differential Equations, Elsevier, Amsterdam, 2006.

\bibitem{5} I. Iscan, A new generalization of some integral inequalities for 
$(\alpha ,m)$-convex functions, Mathematical Sciences\textit{,} (2013),
doi:10.1186/2251-7456-7-22.

\bibitem{6} I. Iscan, New estimates on generalization of some integral
inequalities for $(\alpha ,m)$-convex functions, Contemporary Analysis and
Applied Mathematics, 1 (2) (2013) 253-264.

\bibitem{7} I. Iscan, New estimates on generalization of some integral
inequalities for $s$-convex functions and their applications, International
Journal of Pure and Applied Mathematics, 86 (4) (2013) accepted.

\bibitem{8} A.P. Ji, T.Y. Zhang and F. Qi, Integral inequalities of
Hermite-Hadamard type for $(\alpha ,m)$-GA-convex functions,
arXiv:1306.0852. Available online at http://arxiv.org/abs/1306.0852.

\bibitem{9} Y. Shuang, H.-P. Yin, and F. Qi, Hermite-Hadamard type integral
inequalities for geometric-arithmetically $s$-convex functions, Analysis
(Munich) 33 (2) (2013), 197-208. Available online at
http://dx.doi.org/10.1524/anly.2013.1192.

\bibitem{10} C. P. Niculescu, Convexity according to the geometric mean,
Math. Inequal. Appl. 3 (2) (2000), 155-167. Available online at
http://dx.doi.org/10.7153/mia-03-19.

\bibitem{11} C. P. Niculescu, Convexity according to means, Math. Inequal.
Appl. 6 (4) (2003), 571-579. Available online at
http://dx.doi.org/10.7153/mia-06-53.

\bibitem{12} M.E. Ozdemir, M. Avci, H. Kavurmaci, Hermite-Hadamard type
inequalities for $s$-convex and $s$-concave functions via fractional
integrals, arXiv:1202.0380v1.

\bibitem{13} J. Park, Generalization of some Simpson-like type inequalities
via differentiable $s$-convex mappings in the second sense, International
journal of Mathematics and Mathematical Sciences, vol. 2011, Article ID
493531, 13 pages, doi:10.1155/493531.

\bibitem{14} T.-Y. Zhang, A.-P. Ji and F. Qi, On Integral Inequalities of
Hermite-Hadamard Type for $s$-Geometrically Convex Functions, \textit{%
Abstract and Applied Analysis}, \textbf{2012} (2012), Article ID 560586, 14
pages, doi:10.1155/2012/560586.

\bibitem{15} M.Z. Sar\i kaya and N. Aktan, On the generalization of some
integral inequalities and \ their applications, Mathematical and Computer
Modelling, 54 (2011) 2175- 2182.

\bibitem{16} E. Set, New inequalities of Ostrowski type for mapping whose
derivatives are $s$-convex in the second sense via fractional integrals,
Computers and Math. with Appl. 63 (2012) 1147-1154.

\bibitem{17} M.Z. Sar\i kaya and H. Ogunmez, On new inequalities via
Riemann-Liouville fractional integration, Abstract an Applied Analysis, vol.
2012, Article ID 428983, 10 pages, doi:10.1155/2012/428983.

\bibitem{18} E. Set, M.E. Ozdemir and M.Z. Sar\i kaya, On new inequalities
of Simpson's type for quasi-convex functions with applications, Tamkang
Journal of Mathematics, 43 (3) (2012) 357-364.

\bibitem{19} I. Iscan, Hermite-Hadamard type inequalities for GA-$s$-convex
functions, Le Matematiche, LXIX (2014) Fasc. II, 129-146.

\bibitem{20} I. Iscan, New general integral inequalities for
quasi-geometrically convex functions via fractional integrals, Jurnal of
Inequalities and Applications, 2013, (2013) 491.
\end{thebibliography}
\end{document}